\documentclass[11pt]{amsart}

\usepackage{amsmath,amsthm,amssymb,mathrsfs,amsfonts,verbatim,color}
\usepackage[all,tips]{xy}
\usepackage{etoolbox} 
\usepackage{bbm}
\usepackage{graphicx,ifpdf}
\ifpdf
\fi

\def\G{          \mathcal G}

\def\A{          \mathcal A}

\newcommand{\NN}{{\mathbb N}}
\newcommand{\RR}{{\mathbb R}}
\newcommand{\TT}{{\mathbb T}}

\newcommand{\ZZ}{{\mathbb Z}}

\newcommand{\SL}{\operatorname{SL}}

\newtheorem{thm}{Theorem}[section]
\newtheorem{lem}[thm]{Lemma}

\newtheorem{coro}[thm]{\sc Corollary}

\newtheorem{prop}[thm]{Proposition}

\theoremstyle{definition}

\theoremstyle{remark}

\newtheorem{rema}[thm]{Remark}

\numberwithin{equation}{section}

\begin{document}

\title{Simultaneous dense and nondense orbits and the space of lattices}

\author{Ronggang Shi \and Jimmy Tseng}
\address{R.S.:  
School of Mathematical Sciences, Tel Aviv University, Tel Aviv
69978, Israel, and School of Mathematical Sciences, Xiamen University, Xiamen 361005, PR China}
\email{ronggang@xmu.edu.cn}
\address{J.T.: School of Mathematics, University of Bristol, Bristol BS8 1TW, U.K.}
\email{j.tseng@bristol.ac.uk}

\thanks{J.T. acknowledges the research leading to these results has received funding from the European Research Council under the European Union's Seventh Framework Programme (FP/2007-2013) / ERC Grant Agreement n. 291147.}
\thanks{
R.S. is supported by NSFC (11201388), NSFC (11271278), ERC starter
grant DLGAPS 279893.
}



\date{}



\begin{abstract}
We show that the set of points nondense under the $\times n$-map on the circle and dense for the geodesic flow 
after we identify the circle with a periodic  horospherical orbit of the modular surface has full Haudorff dimension.  We also show the analogous result for toral automorphisms on the $2$-torus and a diagonal flow.  Our results can be interpreted in number-theoretic terms:  the set of well approximable numbers that are nondense under the $\times n$-map has full Hausdorff dimension.  Similarly, the set of well approximable $2$-vectors that are nondense under a hyperbolic toral automorphism has full Hausdorff dimension.  Our result for numbers is the counterpart to a classical result of Kaufmann.
\end{abstract}

\maketitle

\markright{}

\section{Introduction}


Let $f:Y \rightarrow Y$ be a dynamical system on a set $Y$ with a topology and let $\mathcal{S}$ be a finite subset.  Let $ND(f)$ denote the set of points with nondense orbit under $f$ and $ND(f, \mathcal{S})$ be the set of points whose orbit closures miss the subset $\mathcal{S}$.  Let $D(f)$ denote the set of points with dense orbits.   Given another map $\tilde{f}:Y \rightarrow Y$, one could ask for the size of the set $ND(f) \cap D(\tilde{f})$ or, more specifically, $ND(f, \mathcal{S}) \cap D(\tilde{f})$.\footnote{The proofs in~\cite{BET} work for $ND(f, \mathcal{S}) \cap D(\tilde{f})$, not just $ND(f) \cap D(\tilde{f})$.}  Such a question was first asked by V.~Bergelson, M.~Einsiedler, and the second-named author for commuting toral automorphisms and endomorphisms and commuting Cartan actions on certain cocompact homogeneous spaces in~\cite{BET}.  In this paper, we study this same question for certain noncompact phase spaces.  One of our results (Corollary~\ref{cor;gauss}) is even an example of a noncommuting pair of maps.  Consequently, the two main restrictions to the proof technique in~\cite{BET} can be overcome, at least in these special cases.  (Another example for the noncommuting case is the main result in~\cite{Ts14}, in which a different proof technique is used.  See also~\cite{BM}.)  For further details of the history of this question, see~\cite{BET}.  Note that our two main results have consequences for number theory (Corollary~\ref{coro3}), which allows us to complement a classical result of Kaufman~\cite{Kau} (see Remark~\ref{remaOpposite}).
\subsection{Statement of results}
Let $X:=X_d:=\SL_d(\mathbb R)/\SL_d(\mathbb Z)$, the space of unimodular lattices in $\RR^d$.  This space is noncompact and has finite Haar meaure, which we normalize to be a probability measure and denote by $\mu_X$. 
Let 
\begin{equation}\label{eq;h}
h_d: \mathbb R^{d-1}\to X_d \quad \mbox{be defined by}\quad
h_d(s)=
\left (
\begin{array}{cc}
1_{d-1} & s \\
0 & 1
\end{array}
\right)\mod \SL_d(\mathbb Z).
\end{equation}
For   $g\in \SL_d(\mathbb R)$, the left translation by $g$ on $X_d$ is also denoted by $g$. 
Let $\TT^d:=\RR^d/\ZZ^d$ denote the $d$-dimensional torus, $\TT$ denote the circle, and $[\cdot]: \RR^d \rightarrow \TT^d$ denote the natural projection map.  Let ${\mathcal O^{\mathbb N}_{ f}( y) }$
(where $\mathbb N=\{0, 1, 2,\ldots\}$ in this paper) denote the forward orbit of the point $y$ under $f$ and ${\mathcal O^{\mathbb Z}_{ f}( y) }$ denote the forward and backwards orbit.  Our two main results are the following two theorems.

\begin{thm}\label{thm1}
Let  $n\in \mathbb Z_{\ge 2}$,   $T:=T_ n$ be the $\times n$-map  on $\mathbb T$ and
 let $\mathcal{S}=\{s_1, \cdots, s_m\}$ be a finite set of points in $\mathbb R$.   Let 
 $t_0>0$ and
$
g=g_{t_0}:=\mathrm{diag}(e^{t_0}, e^{-t_0})\in \SL_2(\mathbb R)
$.
Then the set \begin{equation}\label{eq;thm1}
\{s\in \mathbb R:
\overline{\mathcal O^{\mathbb N}_{ T}( [s]) }  \cap \mathcal{S}=\emptyset \quad \mbox{and}\quad
 \overline{ \mathcal O_g^{\mathbb N}(h_2(s)) }=X_2  \}
\end{equation}
has full Hausdorff dimension.
\end{thm}

\begin{coro}\label{cor;gauss}
Let the assumptions be as in  Theorem \ref{thm1} and let $\G: \mathbb T\to \mathbb T$ be the Gauss map.
Then the set \[
\{s\in \mathbb R:
\overline{\mathcal O^{\mathbb N}_{ T}( [s]) }  \cap \mathcal{S}=\emptyset \quad \mbox{and}\quad
 \overline{ \mathcal O_\G^{\mathbb N}([s] })=\mathbb T  \}
\]
has full Hausdorff dimension.
\end{coro}
\begin{proof}
Suppose that $s\in \mathbb R$ and 
\begin{equation}
 \overline{ \mathcal O_g^{\mathbb N}(h(s)) }=X_2 .
\end{equation}
It follows from \cite[Theorem 1.7]{efs11} that the continued fraction expansion of $s$
contains all patterns. Therefore $\overline{ \mathcal O_\G^{\mathbb N}([s] })=\mathbb T$.  Thus Corollary \ref{cor;gauss} follows from Theorem \ref{thm1}.
\end{proof}

We say two natural numbers $n, m \geq 2$ are \textit{multiplicatively dependent} if there exists integers $p, q\geq1$ such that $n^p = m^q$.  Otherwise, $n$ and $m$ are \textit{multiplicatively independent}.  For $T_m$, define  \[\widetilde{D}(T_m) :=  \{s\in \RR:  \overline{\mathcal O^{\mathbb N}_{ T_m}( [s])} = \TT  \}.\]  Here, $T_m$ is the $\times m$-map, or, equivalently, the multiplication map on $\TT$ by $m$.
Also for $t>0$, let
\[
 \widetilde D(g_t ,X_d):=\{s\in \mathbb R^{d-1}: \overline{ \mathcal O_{g_t}^{\mathbb N}(h_d(s))}=X_{d}\}. 
\]
Our method of the proof of Theorem \ref{thm1} can also be used to obtain

\begin{thm}\label{coro1}  Let the assumptions be as in  Theorem \ref{thm1}.  Let $\{T_m\}$ be all multiplication maps on $\TT$ by natural numbers $m$ which are multiplicatively independent from $n$.  
Let $\{t_i\}$ be a sequence of positive real numbers. 
Then the set \begin{align*}
\{s\in \mathbb R:
\overline{\mathcal O^{\mathbb N}_{ T_n}( [s]) } \cap  \mathcal{S}=\emptyset\} \cap
\big( \cap_m \widetilde{D}(T_m) \big)\cap
\big( \cap_i \widetilde{D}(g_{t_i},X_ 2) \big)
\end{align*}
has full Hausdorff dimension.
\end{thm}

\begin{thm}\label{thm2}
Let $T=T_\alpha$ be an automorphism of $\mathbb T^2$ induced by  the left multiplication of a hyperbolic  matrix $\alpha\in \SL_2(\mathbb Z)$ and  let $\mathcal{S}=\{s_1, \cdots, s_m\}$ be a finite set of points in $\mathbb R^2$.
 Let 
$t_0>0$ and 
 $g=g_{t_0}:=\mathrm{diag}(e^{t_0}, e^{t_0}, e^{-2t_0})\in  \SL_3(\mathbb R)$.
Then the set \begin{equation}\label{eq;thm2}
\{s\in \mathbb R^2:
\overline{\mathcal O^{\mathbb Z}_{ T}( [s]) }  \cap \mathcal{S}=\emptyset \quad \mbox{and}
 \quad \overline{ \mathcal O_{g}^{\mathbb N}(h_3(s)) }=X_3  \}
\end{equation}
has full Hausdorff dimension.
\end{thm}

  Let $T_\alpha$ and $T_\beta$ be two
   commuting hyperbolic automorphisms of $\TT^2$ and $\A := (T_\alpha, T_\beta)$ be the $\ZZ^2$-action on $\TT^2$ that is generated by~$T_\alpha$ and $T_\beta$. 
  A $\ZZ^2$-action $\A'$ on a torus $\TT^{d'}$ is an \textit{algebraic factor} of $\A$ if there is a surjective toral homomorphism $\varphi:\TT^2 \rightarrow \TT^{d'}$ such that $\A' \circ \varphi = \varphi \circ \A$ and is a \textit{rank-one factor} if, in addition, $\A'(\ZZ^2)$ has a finite-index subgroup consisting of the powers of a single map.  For $T_\beta$, let  \[\widetilde{D}(T_\beta) :=  \{s\in \RR^2:  \overline{\mathcal O^{\mathbb N}_{ T_\beta}( [s])} = \TT^2  \}.\]
Our method of the proof of Theorem \ref{thm2} can also be used to obtain
\begin{thm}\label{coro2}Let the assumptions be as in Theorem \ref{thm2}.  Let $\{T_\beta\}$ be all hyperbolic toral automorphisms of the $2$-torus that commute with $T_\alpha$ and such that the algebraic $\ZZ^2$-actions $(T_\alpha, T_\beta)$ are all without rank-one factors. 
Let $\{t_i\}$ be a sequence of positive real numbers. 
 Then the set \[
\{s\in \mathbb R^2:
\overline{\mathcal O^{\mathbb Z}_{ T_\alpha}( [s]) }  \cap \mathcal{S}=\emptyset\} \cap
\big(\cap_\beta \widetilde{D}(T_\beta) \big) \cap
\big(\cap_{i} \widetilde{D}(g_{t_i}, X_3) \big)
\]
has full Hausdorff dimension.
\end{thm}

%
%

Finally, we note that our two main theorems have consequences for number theory.  Recall that a classical object in the theory of Diophantine approximation is the set of well approximable $d$-vectors $WA(d)$, which is defined to be the complement in $\RR^d$ of the set of badly approximable vectors $BA(d)$.  

\begin{coro}\label{coro3}  Let the assumptions be as in Theorems~\ref{coro1} and~\ref{coro2}.
The sets \begin{align*}\{s\in \mathbb R:
\overline{\mathcal O^{\mathbb N}_{ T_n}( [s]) } \cap  \mathcal{S}=\emptyset\} \cap 
\big(\cap_m \widetilde{D}(T_m) \big )
 \cap WA(1)\\ \{s\in \mathbb R^2:
\overline{\mathcal O^{\mathbb Z}_{ T_\alpha}( [s]) }  \cap \mathcal{S}=\emptyset\} \cap
\big( \cap_\beta \widetilde{D}(T_\beta)\big )
\cap WA(2)\end{align*} have full Hausdorff dimension.
\end{coro}

\begin{proof}
By Dani correspondence, $WA(d)$ is the set of $d$-vectors whose corresponding trajectories are not bounded and, thus, a superset of the $d$-vectors whose corresponding trajectories are dense under the diagonal element 
$g$.  The result is now immediate from our two theorems.
\end{proof}

\begin{rema}\label{remaOpposite}
It is a classical result of Kaufman~\cite{Kau} from 1980  (see also Queff\'elec-Ramar\'e~\cite{QR} for extensions of~\cite{Kau}) that the set $BA(1) \cap D(T_n)$ has full Hausdorff dimension; also see~\cite[Section~1.3.3]{HoSc} and~\cite{JoSa}.  
The first assertion in Corollary~\ref{coro3} immediately implies the other mixed case, namely, $ND(T_n) \cap WA(1)$ has full Hausdorff dimension. 
It is known (see e.g.~Pollington \cite{p79}) that the intersection of $WA(1)$ with 
the set of numbers which are non-normal in every base has full Hausdorff dimension.
\end{rema}

\subsection{Ideas in the proof}  The purpose of this paper is to study simultaneous dense and nondense orbits,
 first studied in~\cite{BET}, in the context of a noncompact phase space and, in the case of Corollary~\ref{cor;gauss}, for noncommuting maps.  The basic idea here follows that in~\cite{BET}; namely, form a large closed fractal of nondense orbits, apply measure rigidity to obtain Haar measure (up to a constant), and then use properties of entropy and measures to conclude our full Hausdorff dimension results.  The main difficulty in the cases that we consider in this paper is the possibility of the escape of all the mass, namely that the weak-$*$ limit of measures goes to $0$.   Using results in~\cite{ek12, k12, s12}, we show that not all the mass can escape when the Hausdorff  dimension is large and, thus, we still obtain dense orbits.

\subsubsection*{Acknowledgements}The authors thank the organizers of the Dynamics and Analytic Number Theory Easter School (2014) and Durham University for providing the venue at which this project was initiated and the organizers of the Interactions between Dynamics of Group Actions and Number Theory program (2014) and the Isaac Newton Institute for Mathematical Sciences for providing a venue for working on this project.  J.T. thanks Sanju Velani for discussions.  We also thank Dmitry Kleinbock, Barak Weiss, and the referees for their comments.

\section{Proof of Theorems~\ref{thm1}~and~\ref{thm2}}\label{secProofThm1and2}

We will prove Theorems~\ref{thm1}~and~\ref{thm2} together (deferring the proof of Propositions~\ref{prop;equidistribution}
and~\ref{corFullMeasure} to Sections~\ref{secEquidistribution} and~\ref{secFullMeasure}, respectively).  Let \begin{align*} X&:= X_2 \textrm{ or } X_3,\\ g_t &:= \mathrm{diag}(e^t, e^{-t}) \textrm{ or } \mathrm{diag}(e^t, e^t, e^{-2t}), \\
g & := g_{t_0} \textrm{ where } t_0 \textrm{ is a fixed positive number}, \\ 
 I&:=\mathbb R / \mathbb Z \textrm{ or } \mathbb R^2 / \mathbb Z^2,\\  T &:=T_n \textrm{ or } T_\alpha, \\ \mathcal{O}_T (s)&:= \mathcal O^{\mathbb N}_{ T_n}(s) \textrm{ or } \mathcal O^{\mathbb Z}_{ T_\alpha}( s).\end{align*}  
 The map $h=h_2$ or $h_3$ defined in (\ref{eq;h}) induces a map $I\to X$ all of which we denote by $h$. 
 Since $T_\alpha$ is a hyperbolic matrix, it has two real eigenvalues, one of which has absolute value $>1$.  Let $\lambda_\alpha$ denote this eigenvalue.

Fix a finite set $\mathcal{S}:=\{s_1, \cdots, s_m\}$ of points in $I$.  In $I$, pick a sequence of open balls $U_q$ centered at $0$ with radius $\rightarrow 0$ as $q \rightarrow \infty$.  Define the following closed $T$-invariant set \[E(q) := E_{T,\mathcal{S}}(q):=\{s \in I \mid \mathcal{O}_T(s) \  \cap \  \cup_{i=1}^m( U_q + s_i) = \emptyset\}.\]

\begin{prop}\label{propNondenseWinning}
 The set $\cup_{q} E(q) $ is winning and, therefore, has full Haudorff dimension. 
\end{prop}
\begin{proof}
Apply~\cite[Corollary~1.5]{T4} and~\cite[Theorem~1.1]{BFK}.
\end{proof}

Consequently, $E(q)$ has Hausdorff dimension as close as we like to full Hausdorff dimension, provided we choose $q$ large enough.  Large Hausdorff dimension implies that the topological entropy of $T$ restricted to such $E(q)$ is close to $h_{top}(T)$:

\begin{prop}[Proposition~2.4 of \cite{BET}]\label{propLargeEntropyFromDim}
Let $F \subset I$ be a closed $T$-invariant set.  Then we have that \[h_{top}(T \vert_F) \geq h_{top}(T) - (d - \dim F) \log(\lambda_1)\] where $\lambda_1 := n$ if $T = T_n$ or $\lambda_1:=|\lambda_\alpha|$ if $T= T_\alpha$.
\end{prop}

The variational principle now gives the existence of a $T$-invariant Borel probability measure $\nu:= \nu(q)$, whose support lies in $E(q)$, such that $h_\nu(T\vert_{E(q)})$ is as close to $h_{top}(T)$ as we like, provided  we choose $q$ large enough.  By ergodic decomposition, we may restrict to an ergodic component and thus assume that $\nu$ is ergodic.  

\begin{prop}\label{prop;equidistribution}
Let $\nu$ be a $T$-invariant, ergodic probability measure on $I$ with positive entropy. Then any weak-$*$
limit $\mu$ of \begin{equation}\label{eq;discrete}
\frac{1}{N}\sum_{i=0}^{N-1} (g^i)_*(h_*\nu) \quad \mbox{as}\quad N\to \infty
\end{equation}
is equal to $c\mu_X$ where $\mu_X$ is the probability Haar measure on $X$ and $0\leq c\leq 1$. 
\end{prop}

We prove Proposition~\ref{prop;equidistribution} in Section~\ref{secEquidistribution} using \cite[Theorems 3.1,~5.1~and~6.1]{s12} of the first-named author.  To guarantee that $c \neq 0$, we also need the the following proposition, ensuring that not all of the mass can escape.  The proposition is essentially a corollary of~\cite[Theorem 1.6]{ek12} and~\cite[Theorem 1.3]{k12}.

\begin{prop}\label{prop;nonescape}
Let $\nu$ be a $T$-invariant, ergodic probability measure on $I$ with entropy
$h_\nu(T)= \delta h_{top}(T)$ where $0\le \delta\le 1$. Then any weak-$*$
limit $\mu$ of (\ref{eq;discrete})
will have total mass $\mu(X)\ge (1+\mathrm{dim}(I))\delta-\mathrm{dim}( I)$.
\end{prop}
\begin{proof}
Since \begin{equation}\label{eq;link}
\frac{1}{t_0N}\int_0^{t_0N} (g_t)_*(h_*\nu)\, dt
=\frac{1}{t_0}\int_{0}^{t_0}(g_t)_*\frac{1}{N}\sum_{k=0}^{N-1} (g^k)_*(h_*\nu)~dt, 
\end{equation}
it suffices to show that any weak-$*$
limit $\mu_1$ of
\begin{equation}\label{eq;continuous}
\frac{1}{\tau}\int_0^\tau (g_t)_*(h_*\nu)\, dt \quad \mbox{as}\quad \tau \to \infty
\end{equation}
has total mass $\mu_1(X)\ge (1+\mathrm{dim}(I))\delta-\mathrm{dim}( I)$.

We claim that for $\nu$ almost every $s\in I$ one has
\begin{equation}\label{eq;local}
\lim_{r \searrow 0}  \frac{\log\nu(B(s,r))}{\log r}= \delta \mathrm{dim}( I).
\end{equation}
For $T=T_n$, the claim follows from 
the main theorem of \cite{bk}.
For $T=T_\alpha$ the claim follows from 
  \cite[Lemma 3.2]{y82}.
Therefore for any $\epsilon >0$ (sufficiently small) there exists $r_0>0$  and a subset  $I_\epsilon \subset I$ such that for $0<r\le r_0$ 
we have 
\[
\nu(I_\epsilon)>1-\epsilon\] and 
\begin{align}\label{eqnPowerMeasure}\nu(B(s, r))\le r^{(\delta-\epsilon) \mathrm{dim}(I)}
\quad \mbox{for }s\in I_\epsilon.
\end{align}
Therefore $\frac{1}{1-\epsilon}h_*(\nu|_{I_\epsilon})$ has dimension at least $(\delta-\epsilon) \mathrm{dim}(I)$
in the unstable horospherical direction of $g_1$ according to \cite[ Definition 1.5]{ek12}. 
Therefore 
\cite[Theorem 1.6]{ek12}  and \cite[Theorem 1.3]{k12}
imply that 
\[
\mu_1(X) \ge (1-\epsilon ) (1+\mathrm{dim}(I))(\delta-\epsilon)-\mathrm{dim}( I).
\]
The conclusion follows 
by letting $\epsilon $ decrease  to $0$. 
\end{proof}


We wish to understand the following dense set of forward orbits
 \[D(g, X) :=\{s\in I : \overline{\mathcal O_g^{\mathbb N}(h(s))}=X\}\] in terms of the measure $\nu$ supported on $E(q)$.

\begin{lem}\label{lemmTInvariance}  The set $D(g, X)$ is $T$-invariant.
\end{lem}
\begin{proof}
In the setting of Theorem \ref{thm1}, 
it is proved in \cite[Lemma 4.3]{efs11} that the set
\[
D'(g_t; X):=\{s\in I: \overline {\{g_t h(s): t\ge 0\}}=X\}
\]
is $T$-invariant. 
The same  argument there replacing 
$D'(g_t; X)$ by $D(g, X)$ gives the lemma. 

 In  the setting of Theorem \ref{thm2}, we let  
 \begin{equation}\label{eq;galpha}
g_\alpha=
\left(
\begin{array}{cc}
\alpha & 0\\
0 & 1
\end{array}
\right).
\end{equation}
Since $g_\alpha
\in SL_3(\mathbb Z)$ a simple calculation shows that 
\[g_t (h(\alpha(s)))=g_{\alpha}g_t(h(s))  \] 
from which the lemma follows.  
\end{proof}

\begin{prop}\label{corFullMeasure}
Let $\nu$ be a $T$-invariant and ergodic probability measure on $I$. Let $\mu$ be a weak-$*$
limit  of (\ref{eq;discrete}). 
Suppose that $\mu=c\mu_X$ for some $0<c\le 1$. Then
 
 \[\nu(D(g, X))=1.\] 
\end{prop}

\noindent We prove Proposition~\ref{corFullMeasure} in Section~\ref{secFullMeasure}.

\begin{rema}
If we were to only prove Theorems~\ref{thm1} and~\ref{thm2}, then we would only need for the conclusion of Proposition~\ref{corFullMeasure} to be $\nu(D(g, X))>0$, in which case Lemma~\ref{lemmTInvariance} is not used.  For Theorems~\ref{coro1} and~\ref{coro2}, we need the full conclusion of Proposition~\ref{corFullMeasure} and thus also Lemma~\ref{lemmTInvariance}.  The analogous remark can be made for~\cite[Theorem~3.2]{BET} and~\cite[Lemma~3.1]{BET}.
\end{rema}

\subsection{Finishing the proof of Theorems~\ref{thm1} and~\ref{thm2}}

Apply Propositions~\ref{prop;equidistribution},~\ref{prop;nonescape}, and~\ref{corFullMeasure} to the measures $\nu:= \nu(q)$ supported on $E(q)$, to obtain $\nu(D(g, X))=1$.  Next apply the mass distribution principle to (\ref{eqnPowerMeasure}) with $\delta := \delta(q)<1$ (where $\delta(q) \nearrow 1$ as $q \rightarrow \infty$) to obtain \[\dim(D(g, X) \cap E(q)) \geq (\delta(q)-\epsilon) \mathrm{dim}(E(q)).\]  (We note that, since (\ref{eq;local}) holds with $\dim(I)$ replaced by $\dim(E(q))$, so does (\ref{eqnPowerMeasure}).)  Letting $\epsilon \rightarrow 0$ and taking a union over $q$, we have that \[\dim\big(D(g, X) \cap ND(T)\big) = \dim(I),\] as desired.

\section{Proof of Theorems~\ref{coro1} and~\ref{coro2}}\label{secProofCoro1and2} Let $T_m$ be as in Theorem~\ref{coro1} and let 
$\nu$ be as in the proof of Theorem \ref{thm1}.   We have $\nu(D(T_m))=1$ by~\cite[Theorem~3.2] {BET} and thus 
\[
\nu\bigg(\cap_i D(g_{t_i} , X_2) \cap \big(\cap_{m} D(T_m)\big) \bigg)=1.
\]
  The rest of the proof of Theorem~\ref{coro1} is exactly the same as that  for the Theorem \ref{thm1}.  Replacing $T_m$ with $T_\beta$ from Theorem~\ref{coro2} gives the proof for Theorem~\ref{coro2}.

\section{Proof of Proposition~\ref{prop;equidistribution}}\label{secEquidistribution}


We claim that any weak-$*$ limit  $\mu_1$ of (\ref{eq;continuous})
is equal to $c\mu_X$. We first prove the proposition using this claim  and then give its proof in the rest of this section.

Let $\mu $ be a weak-$*$ limit of  (\ref{eq;discrete}). 
If $\mu=0$ there is nothing to prove. Assume that $\mu(X)=c>0$.
It follows from (\ref{eq;link}) that 
\[
\mu_2:=\frac{1}{t_0}\int_0^{t_0} (g_t)_* \mu\, dt
\]
is a weak-$*$ limit of (\ref{eq;continuous}). 
The claim above implies that $\mu_2=c\mu_X$. On the other hand
\cite[Theorem 5.27]{elw} implies
\begin{equation}\label{eq;entropy}
h_{c^{-1}\mu_2}(g)=\frac{1}{t_0}\int_0^{t_0}  h_{c^{-1}(g_t)_*\mu}(g)\, dt. 
\end{equation}
It follows from \cite[Corollary 7.10]{el10} that 
\[
h_{c^{-1}(g_t)_*\mu}(g)=h_{c^{-1}\mu}(g)\le h_{\mu_X}(g)
\]
and that equality holds if and only if $\mu$ and hence $(g_t)_*\mu$ is 
equal to $c\mu_X$. Therefore (\ref{eq;entropy}) implies that $\mu=c\mu_X$
which completes the proof of the proposition. 

Now we prove the claim in the setting of Theorem \ref{thm1}.  The natural map
\[
\pi:  \SL_2(\mathbb R) \to PGL_2(\mathbb R)
\]
induces  a diffeomorphism of homogeneous spaces 
\[
X_2\to \widetilde{X}_2:= PGL_2(\mathbb R)/PGL_2(\mathbb Z). 
\]
Let $n=p_1^{\sigma_1}\cdots p_k^{\sigma_k}$ be the prime decomposition of $n$. 
Let 
\[
L=\prod _{i=1}^k PGL_2(\mathbb Q_{p_i}), \quad
\Gamma=PGL_2(\mathbb Z[\frac{1}{p_1}, \ldots, \frac{1}{p_k}]) \quad \mbox{ and }\quad 
K=\prod _{i=1}^k PGL_2(\mathbb Z_{p_i}). 
\]
Then $Y=PGL_2(\mathbb R)\times L/\Gamma$  where 
$\Gamma $ embeds 
diagonally is a finite volume homogeneous space. The natural map
\[
\eta: Y\to X _2,
\]
which maps $(\pi(g), h)\mod \Gamma$,  where $g\in \SL_2(\mathbb R)$ and $h\in K$, to 
$g\mod \SL_2(\mathbb Z)$,  has compact fibers. 

Let 
\[
a=\left( 
\begin{array}{cc}
n & 0 \\
0 & 1
\end{array}
\right) ^{k+1}\in PGL_2(\mathbb R)\times L. 
\]
It follows from \cite[Lemma 5.3]{s12} that there exists an $a$-invariant and ergodic probability measure
$\tilde \nu$ on $Y$ such that $\eta_*\tilde\nu=h_*\nu$ and $h_{\tilde \nu}(a)>0$.
Let  $\mu_1$ be a weak-$*$ limit of (\ref{eq;continuous}) for the sequence  $\{\tau_i:\in \mathbb N\}$.
By possibly passing to a subsequence we assume that 
\[
\tilde \mu_1 :=\lim_{i\to \infty}\frac{1}{\tau_i}\int_0^{\tau_i} (\pi(g_t))_*((\pi\circ  h)_*\nu)\ dt
\]
exists. 
If $\mu_1(X)=0$ there is nothing to prove. 
So we assume that $\mu_1(X)=c>0$.
Since $\eta$ has compact fibers we have $\tilde \mu_1 (Y)=c>0$. 
It follows from \cite[Theorem 3.1]{s12} that 
all the ergodic components of $\tilde\mu_1$ with respect to the group generated 
by $a$ have positive entropy. 
Therefore \cite[Theorem 1.1]{l06} implies that $\tilde \mu_1 =c\mu_Y$. 
Since $\eta\tilde \mu_1=\mu_1$, we have 
\[
\mu_1=c\mu_{X_2}
\]
which completes the proof of the claim in the setting of Theorem \ref{thm1}. 

Now we prove the claim in the setting of Theorem \ref{thm2}. 
Let $\mu_1$ be a weak-$*$ limit of (\ref{eq;continuous}). 
Without loss of generality we assume that $\mu_1(X_3)=c>0$. 
It follows from \cite[Theorem 3.1]{s12} that 
all the ergodic components of $\nu$ with respect to the group generated 
by $g_{\alpha}$, where $g_\alpha$ is defined in (\ref{eq;galpha}),  have positive entropy.  Therefore \cite[Corollary 1.4]{ekl06}
implies that $\mu_1=c\mu_{X_3}$ which completes the proof of the claim and 
hence the proposition.

\section{Proof of Proposition~\ref{corFullMeasure}}\label{secFullMeasure}


The proof is adapted from the proof of~\cite[Theorem~3.2]{BET}. 
Let $d$ denote the distance function on $X$.  Fix a countable dense subset $\{x_1, x_2, \cdots\} \in X$.  Inductively define the following disjoint sets: \begin{align*}
ND(1,1) &:= \bigg{\{}s \in I :  d\bigg(\overline{\{g^kh(s)\}_{k\in \mathbb N}}, x_1)\bigg) \geq 1\bigg{\}} \\
ND(i,n)  &:= \bigg{\{}s \in I :  d\bigg(\overline{\{g^k h(s)\}_{k\in \mathbb N}}, x_i)\bigg) \geq \frac1 n\bigg{\}} \backslash \bigcup_{(j,m) < (i,n)} ND(j,m).
\end{align*}  (Here we have fixed the total ordering  $<$ of $\NN \times \NN$ given by $(i,n)<(i',n')$ if either $i+n<i'+n'$ or $i+n= i'+n'$ and $i<i'$.)  It follows that $ND:=I \backslash D(g, X)$ is the union of these sets. 

Now assume that the conclusion is false; this, by Lemma~\ref{lemmTInvariance}, is equivalent to $\nu(ND)=1$.  
Decompose $\nu= \sum_{i,n}\nu_{i,n}$ where  $\nu_{i,n}$ is the restriction of $\nu$ to $ND(i,n)$.  
 Using  Banach--Alaoglu theorem, from the subsequence used to obtain $\mu$, we may extract
 a further subsequence $N_\ell$
  such that for all $(i, n)$ the
 measure
  \[\frac{1}{N_\ell}\sum_{k=0}^{N_\ell-1} (g^k)_*(h_*\nu_{i,n})\]
  converges in the weak-$*$ topology to $\mu_{i,n}$.
  We note that the $\mu_{i,n}$ are $g$-invariant measures on $X$ and we have $\mu = \sum_{i,n} \mu_{i,n}$ by the disjointness of the sets $\{ND(i,n)\}$.  
 
 Let $B^M(y, r)$ denote an open ball in a metric space $M$ around $y \in M$ of radius $r>0$.  By construction, we have that $\nu_{i,n}$-a.e. point $s$ in $I$ satisfies \begin{align*}
h(s) \not \in g^{-k} B^X(x_i, 1/n)   \end{align*} for all $k \geq 0$.   Hence, \[\nu_{i,n}(h^{-1} g^{-k} B^X(x_i, 1/n) )  = (g^k)_*h_* \nu_{i,n}(B^X(x_i, 1/n)) = 0.\]  Since this holds for all $k \geq 0$, we have \begin{align*}
\mu_{i,n}\left(B^X(x_i, 1/n)\right) = 0.\end{align*}  Finally, since $\mu_{i,n}$ is $g$-invariant, we have that \begin{align*}
\bigcup_{k=0}^\infty g^{-k}B^X(x_i, 1/n)  \end{align*} is a $\mu_{i,n}$-null set, but  it is also a full Haar measure set by the ergodicity of $g$.  Consequently, $\mu_{i,n}$ is singular to Haar measure.  Now, using the fact that a finite measure which is an infinite sum of finite measures, each of which is singular to $\mu_X$, is singular to $\mu_X$ itself, we contradict the assumption that $\mu=c\mu_X$ for some $c>0$. This completes the proof.  

\begin{thebibliography}{99}


\bibitem{BET}
V. Bergelson, M. Einsiedler and J. Tseng, {\em
Simultaneous dense and nondense orbits for commuting maps,}
to appear in Israel J. Math.

\bibitem{bk}
M. Brin and A. Katok, {\em On local entropy}, in Geometric Dynamics (Rio de Janeiro, 1981),
 Lecture Notes in Math. 1007, Springer, Berlin, 1983.


\bibitem{BFK} R. Broderick, L. Fishman and D. Kleinbock, {\em Schmidt's game, fractals, and orbits of toral endomorphisms,} Ergodic Theory Dynam. Systems {\bf 31} (2011), 1095-1107.




\bibitem{efs11}
M. Einsiedler, L. Fishman and  U. Shapira,
{\em Diophantine approximations on fractals,} Geom. Funct. Anal. 21 (2011)
14-35.

\bibitem{ek12}
M. Einsiedler and S. Kadyrov, 
{\em Entropy and escape of mass for $\SL_3(\mathbb Z)
\backslash \SL_3(\mathbb R)$,} Israel J. Math. 190 (2012), 253-288.

\bibitem{ekl06}
M. Einsiedler, A. Katok and E. Lindenstrauss, {\em Invariant measures and the set
of exceptions to Littlewood's conjecture}, Ann. of Math. {164} (2006), 513-560.

\bibitem{el10}
M. Einsiedler and E. Lindenstrauss, {\em
Diagonal actions on locally homogeneous spaces. Homogeneous flows,}
in  moduli spaces and arithmetic, Clay Math. Proc., 10, Amer. Math. Soc., Providence, RI, 2010,
 155-241.

\bibitem{elw}
M. Einsiedler, E. Lindenstrauss and T. Ward, 
{\em Entropy in Dynamics,} book in preparation. 



\bibitem{HoSc} M.~Hochman and P.~Shmerkin, {\em Equidistribution from fractals,} preprint at http://arxiv.org/abs/1302.5792, (2013).


\bibitem{JoSa} T.~Jordan and T.~Sahlsten, {\em Fourier transforms of Gibbs measures for the Gauss map,} preprint,  arXiv:1312.3619v2 (2013).

\bibitem{k12}
S. Kadyrov,
{\em Entropy and escape of mass for Hilbert modular spaces,}  J. Lie Theory 22 (2012), no. 3, 701-722.

\bibitem{Kau} R. Kaufman, {\em Continued fractions and Fourier transforms,} Mathematika { 27} (1980), 262-267.


%
\bibitem{l06}
E. Lindenstrauss, 
{\em
Invariant measures and arithmetic quantum unique ergodicity,
}
Ann. of Math. 163 (2006) 165-219.


\bibitem{BM} B. Lytle and A. Maier, {\em Simultaneous dense and nondense orbits for noncommuting toral endomorphisms}, preprint (2014).

\bibitem{p79}
A. D. Pollington, 
{\em The Hausdorff dimension of a set of non-normal well approximable numbers}, Number Theory, Carbondale (1979), Lectures Notes in Math. 751, Springer-Verlag, Berlin, 1979, 256-264.


\bibitem{QR} M. Queff\'elec and O. Ramar\'e,
{\em Analyse de Fourier des fractions continues \`a quotients restreints,
} Enseign. Math., {49} (2003), 335-356.


\bibitem{s12}
R. Shi,{\em  Convergence of measures under diagonal actions on homogeneous spaces,
} Adv. Math. 229 (2012), 1417-1434.

\bibitem{T4} J. Tseng, {\em Schmidt games and Markov partitions,} Nonlinearity {22} (2009), 525-543. 

\bibitem{Ts14} J. Tseng, {\em Simultaneous dense and nondense orbits for toral automorphisms}, preprint, arXiv:1406.1970 (2014).

\bibitem{y82}
L.-S. Young, 
{\em Dimension, entropy and Lyapunov exponents,} Ergodic Theory Dynam.
Systems 2 (1982), 109-124




\end{thebibliography}
\end{document}